\documentclass[12pt,a4paper]{amsart}
\usepackage{amsfonts}
\usepackage[active]{srcltx}

\usepackage{enumerate}
\usepackage[notref,notcite]{showkeys}
\usepackage{amsmath,amssymb,xspace,amsthm}

\newtheorem{theorem}{Theorem}[section]
\newtheorem{example}{Example}[section]
\newtheorem{remark}[theorem]{Remark}
\newtheorem{lemma}[theorem]{Lemma}
\newtheorem{proposition}[theorem]{Proposition}
\newtheorem{corollary}[theorem]{Corollary}

\newtheorem{definition}[theorem]{Definition}
\numberwithin{equation}{section} \theoremstyle{definition}

\def\k{\kappa}

\def\span{\operatorname{span}}

\newcommand{\ZZ}{{\mathbb Z}}

\newcommand{\C}{\ensuremath{\mathbb C}\xspace}

\renewcommand{\a}{\ensuremath{\alpha}}
\renewcommand{\l}{\ensuremath{\lambda}}

\newcommand{\Z}{\ensuremath{\mathbb{Z}}\xspace}
\newcommand{\N}{\ensuremath{\mathbb{N}}\xspace}

\newcommand{\W}{\ensuremath{W}\xspace}

\newcommand{\rank}{\operatorname{rank}}

\renewcommand{\phi}{\varphi}

\renewcommand{\leq}{\leqslant}

\newcommand{\K}{\ensuremath{\mathcal{\K}}\xspace}

\def\mg{\mathfrak{g}}
\def\mh{\mathfrak{h}}

\def\sl{\mathfrak{sl}}
\def\gl{\mathfrak{gl}}

\def\ot{\otimes}
\def\l{\lambda}

\def\K{\mathcal{K}}
\def\F{\mathcal{F}}
\def\T{\mathcal{T}}
\def\L{\mathfrak{L}}

\def\span{\text{span}}

\def\End{\text{End}}

\begin{document}
\title[Simple Witt modules]{Simple Witt modules that are finitely generated over the cartan subalgebra}
\author{Xiangqian Guo, Genqiang Liu, Rencai Lu  and Kaiming Zhao}
\date{}\maketitle
\begin{abstract}  Let $d\ge1$ be an integer, $W_d$ and  $\K_d$ be the Witt  algebra and the weyl algebra over the Laurent polynomial  algebra $A_d=\mathbb{C} [x_1^{\pm1}, x_2^{\pm1}, ..., x_d^{\pm1}]$, respectively. For any    $\gl_d$-module $M$   and any admissible  module $P$ over the extended Witt algebra $\widetilde W_d$, we define a $W_d$-module structure on the tensor product $P\otimes M$.     We prove in this paper that  any  simple $W_d$-module that is  finitely generated over the cartan subalgebra is a quotient module of the
$W_d$-module $P \otimes M$ for a finite dimensional simple  $\gl_d$-module $M$ and a   simple $\K_d$-module $P$ that are finitely generated over the cartan subalgebra. We also characterize all    simple $\K_d$-modules and all   simple admissible $\widetilde W_d$-modules that are finitely generated over the cartan subalgebra.

\end{abstract}

\vskip 5pt
\noindent
{\em 2010  Math. Subj. Class.:}
17B10, 17B20, 17B65, 17B66, 17B68

\vskip 10pt

\section{Introduction}

We denote by $\mathbb{Z}$, $\mathbb{Z}_+$, $\mathbb{N}$ and
$\mathbb{C}$ the sets of  all integers, nonnegative integers,
positive integers and complex numbers, respectively.  All algebras and vector spaces are assumed to be over $\C$.
For any vector space $V$, we denote by $V^*$ the dual space of $V$,
and for any subset $S$ of some abelian group, we denote $S^\star=S\setminus\{0\}$.
Let $d\ge1$ be a fixed integer throughout this study.

Representation theory of Lie algebras is a rich  topic atracting the extensive attention from many mathematicians. Classification of simple modules is an important step in the study of a module category
over an algebra.
Let $\mg$ be a  Lie algebra with a Cartan subalgebra $\mh$. A $\mg$-module $M$ is called  a weight module if the action of
$\mh$ on $M$ is diagonalizable. The other extreme is  that a $\mg$-module $M$ is called $U(\mh)$-torsion-free  if for any nonzero $v\in M$ there is non nonzero $g\in U(\mh)$ such that
$gv=0$. In this paper we will show that simple $W_d$-modules that are finitely generated over $\mh$ are $U(\mh)$-torsion-free.

The classification of simple weight modules for several classes of Lie algebras has been achieved over many years of many mathematicians' endeavor. Here we only mention a few such achievements related to the study of the present paper. 
Finite dimensional simple modules for  finite-dimensional semisimple
Lie algebras were classified by Cartan in 1913, see \cite{Ca}. The classification of simple Harish-Chandra modules over the Virasoro algebra was completed by Mathieu in 1992, see [M1].
 The classification of simple weight
modules over finite dimensiona semisimple Lie algebras with finite-dimensional weight spaces  was obtained in 2000,
see \cite{M2}. Besides $\sl_2$ (and its some deformations), all simple weight modules were constructed only for the aging algebra \cite{LMZ}, 
 the Schrodinger algebra \cite{BL2}, and 
 the Euclidean algebra \cite{BL1}.

The study on $U(\mh)$-torsion-free  $\mg$-modules   has just started with rank one a few years ago, see \cite{N1,N2,TZ1} where    simple   $\mg$-modules that are free $U(\mh)$-module of rank one were classified for  finite dimensional simple
Lie algebras and for Witt algebras $W_d$ and $W^+_d$. This category is not an abelian category. Now we turn to investigate  the category of $\mg$-modules which are finitely generated when
restricted to $U(\mh)$. We consider the Lie algebra $W_d$ of vector fields on an $d$-dimensional torus, that   is, the derivation Lie algebra of the Laurent
polynomial algebra $A_d=\C[x_1^{\pm1},x_2^{\pm1},\cdots,
x_d^{\pm1}]$. The algebra $W_d$ is a natural higher rank
generalization of the Virasoro algebra, which has many applications
to different branches of mathematics and physics (see \cite{M2,
L1,L2,L3,L4,L5}) and at the same time a much more complicated
representation theory.

Over the last two decades, the weight representation theory of Witt algebras was  extensively
studied by many algebraists and physicists; see for example \cite{B, E1, E2, BMZ, GLZ,L3, L4, L5,LZ,LLZ,
MZ2,Z}. In 1986,  Shen defined a class of modules $F^\alpha_b(V)$
over the Witt algebra $W_d$ for  $\a\in\C^d$, $b\in\C$,
and a simple  module $ V$ over the special linear Lie algebra
$\sl_d$, see \cite{Sh}, which were also given by Larsson in 1992,
see \cite{L3}. In 1996, Eswara Rao determined the necessary and
sufficient conditions for   these modules to be irreducible when $V$
is finite dimensional, see \cite{E1, GZ}. Very recently,  Billig and Futorny
\cite{BF}   proved that simple
$W_d$-modules    with finite-dimensional weight spaces are modules of the highest weight type or simple quotient 
  modules from  $F^\alpha_b(V)$.

In the present paper, for  a $\gl_d$-module $V$ and an admissible $\widetilde{W}_d$-module $P$, we define a $W_d$-module $\mathcal{F}(P, V)$ which
generalizes the construction of $F^\alpha_b(V)$. Since there exists a natural
algebra homomorphism from $U(\widetilde{W}_d)$ to the Weyl algebra $\K_d$, each $\K_d$-module can be viewed as an admissible $\widetilde{W}_d$-module.
Let $\C^d$ be the natural $d$-dimensional representation of $\gl_d$ and let $V(\delta_k,k)$ be its
$k$-th exterior power, $k = 0,\cdots,d$. In the paper \cite{LLZ}, it was shown that  when $V$ is a weight module, $\mathcal{F}(P, V)$ is a simple module over
$W_d$ if and only if  $M\not\cong V(\delta_k, k)$ for any $k\in \{0, 1,\cdots, d\}$.

For any $k\in \{0, 1,\cdots, d\}$,
there are $W_d$-module homomorphisms
\begin{equation*}\begin{array}{lrcl}
\pi_{k-1}:& \F(P,V(\delta_{k-1},k-1)) & \rightarrow & \F(P, V(\delta_{k},k)),\\
      & p\otimes v & \mapsto & \sum_{j=1}^{k} D(e_j,0)p\otimes e_j\wedge v,
\end{array}\end{equation*}
for all $p\in P$ and $v\in V(\delta_{k-1}, k-1)$ where $\F(P,V(\delta_{-1},-1)) =0$.
Let $\tilde \L_d(P,k)=\text{Ker} \pi_{k}$ and $ \L_d(P,k)=\text{Im}  \, \pi_{k-1}$.
Then from Theorem 3.5 in [LLZ], the $W_d$-modules $ \L_d(P,k)$ are simple for $k=1,2,\ldots,d$. Moreover,
$$0\subseteq  \L_d(P,k)\subseteq \tilde \L_d(P,k)\subseteq  \F(P, V(\delta_k,k)).$$

In the present paper, we show that if $M$ is   a simple $W_d$-module that is  finitely generated over $\mh$,
then $M$ is a simple quotient of a $W_d$-module $\F( P, V)$ for  a finite dimensional  simple $\gl_d$-module $V$ and a  simple  $\K_d$-module  that is  finitely generated over $\mh$.

The paper is organized as follows. In Section 2, we recall the definition of the Witt algebra $W_d$, the entended Witt algebra $\widetilde{W}_d$, the Wyel algebra $\K_d$, and give the construction of the $W_d$-module $\F( P, V)$ and some of its properties (Propositions 2.2 and Theorem 2.3). In Section 3,
we generalize  the weighting functor $\mathfrak{W}$ introduced in \cite{N2} for finite dimensional simple Lie algebras to the Witt module category  in a slightly different way and discuss some of its properties (Propositions \ref{p3.5}).
 In Section 4, we give the classification of  simple admissible $\widetilde{W}_d$-modules  that are  finitely generated over $\mh$  (Theorem \ref{main1}), and the classification of  simple admissible $\widetilde{W}_d$-modules   that  are  free $U(\mh)$-module of finite rank (Corollary \ref{cor4.8}).  In Section 5, using the technique of covering module established in \cite{BF},
 we complete the classification of  simple $W_d$-modules     that are  finitely generated over $\mh$ (Theorem \ref{thm5.5}). We also show that
 $\L_d(P,i)$ are not   free $U(\mh)$-modules   for $ i=2,3,\ldots,d $.  Using this fact, we obtain the  classification of simple $W_d$-modules   that are free  $U(\mh)$-modules of finite rank (Corollary \ref{cor5.5}).
 In Section 6,  we give a description of  simple $\K_d$-modules   that  are finitely generated  over $\mh$ (Lemma 6.1). We also give an example which shows that there exists a simple $W_d$-module that is a free $U(\mh)$-module of any given positive rank.
 Other main techniques we use in this paper are  Quillen-Suslin Theorem and other results from commutative algebra \cite{L}.

\section{$\W_d$-modules from $\sl_d$-modules}

As usual,   $\ZZ^{d}$ (and other similar notations)
denotes the direct sum of $d$ copies of the additive group $\ZZ$, and we consider it as the additive group of all column vectors with integer entries. For
any $a=(a_1,\cdots, a_d)^T \in \Z_+^d$ and $n=(n_1,\cdots,n_d)^T
\in\C^d$, we denote $n^{a}=n_1^{a_1}n_2^{a_2}\cdots n_d^{a_d}$,
where $T$ means taking transpose of the matrix.
Let $\gl_d$ be the Lie algebra of all $d \times d$ complex matrices, and $\sl_d$ be the
subalgebra of $\gl_d$ consisting of all traceless matrices.  For $1
\leq i, j \leq d$, we use $E_{ij}$ to denote the matrix with $1$
at the $(i, j)$ entry and zeros elsewhere.
 We know that
$$\gl_d=\sum_{1\leq i, j\leq
n}\C E_{i,j}.$$

Let $\mh=\span\{h_{i}\,|\,1\le i\le d-1\}$ where
$h_i=E_{ii}-E_{i+1,i+1}$. 
Let $\Lambda^+=\{\l\in\mh^*\,|\,\l(h_i)\in\Z_+ \text{ for } i=1,2,...,d-1\}$ be the set of dominant weights with respect to $\mh$. For any
$\psi\in \Lambda^+$,  let $V(\psi)$ be  the simple $\sl_d$-module with
highest weight $\psi$. We make $V(\psi)$ into a $\gl_d$-module by
defining the action of the identity matrix $I$ as some scalar
$b\in\mathbb{C}$. We denote the resulting $\gl_d$-module as $V(\psi,b)$.

We fix the vector space $\mathbb{C}^d$ of $d\times 1$ matrices.
Denote its standard basis by $\{e_1,e_2,...,e_d\}$. Let
$(\,\cdot\,|\, \cdot\, )$ be the standard symmetric bilinear form  on $\mathbb{C}^d$
such that $(u | v)=u^Tv\in\mathbb{C}$.

Define the fundamental weights $\delta_i\in\mh^*$ by
$\delta_i(h_j)=\delta_{i,j}$ for all $i,j=1,2,..., d-1$. It is
well known that the module $V(\delta_1, 1)$ can be realized as the
natural representation of $\gl_d$ on $\mathbb{C}^d$ (the matrix
product), which we can write as $E_{ji}e_l=\delta_{li}e_j$.  In
particular,
\begin{equation}(ru^T)v=(u|v)r,\,\,\forall\,\, u,v,r\in \mathbb{C}^d.\end{equation}
The exterior
product $\bigwedge^k(\mathbb{C}^d)=\mathbb{C}^d\wedge\cdots\wedge
\mathbb{C}^d\ \ (k\ \mbox{times})$ is a $\mathfrak{gl}_d$-module
with the action given by $$X(v_1\wedge\cdots\wedge
v_k)=\sum\limits_{i=1}^k v_1\wedge\cdots v_{i-1}\wedge
Xv_i\cdots\wedge v_k, \,\,\forall \,\, v_i\in  \mathbb{C}^d,  X\in \gl_d,$$ and the following
$\gl_n$-module isomorphism is well known:
\begin{equation}\label{dk}{\bigwedge}^k(\mathbb{C}^d)\cong V(\delta_k,
k),\,\forall\,\, 1\leq k\leq d,\end{equation}
where $V(\delta_d,
d)$ is a $1$-dimsional module.
For convenience of late use, we let $V(\delta_0,
0)$ be the 1-dimensional trivial $\gl_d$-module. We set  $\bigwedge^0(\mathbb{C}^d)=\C$ and
$v\wedge a=av$ for any $v\in\C^d, a\in\C$.

\subsection{ Witt algebra $W_d$}
We denote by  $ W_d$ the derivation Lie algebra of the
Laurent polynomial algebra $A_d=\C[x_1^{\pm1}, \cdots,x_d^{\pm1}]$. 
Set $\partial_i=x_i\frac{\partial}{\partial x_i}$ for $i=1,2,\dots,d$ and 
$x^r=x_1^{r_1}x_2^{r_2}\cdots x_d^{r_d}$ for $r=(r_1,r_2,\cdots, r_d)^T\in\mathbb{Z}^d$.
For $u=(u_1,\cdots, u_d)^T \in \mathbb {C}^d$ and $r\in \mathbb{Z}^d$, we denote
$$D(u,r)=x^r\sum_{i=1}^du_i\partial_i\in\W_d.$$ Then we have the Lie bracket
$$[D(u,r),D(v,s)]=D(w,r+s),\ \forall\ u,v\in \mathbb {C}^d, r,s\in \mathbb {Z}^d,$$
where $w=(u | s)v-(v | r)u$. Note that for any $u,v,\xi,\eta\in
\mathbb{C}^d$, both $uv^T$ and $\xi\eta^T$ are $d\times d$ matrices, and
 \begin{equation*}(uv^T)(\xi\eta^T)=(v|\xi)u\eta^T.\end{equation*}
 We know that $\mh=\span\{\partial_1, \partial_2, ... , \partial_d\}$ is the Cartan
 subalgebra of $\W_d$.

It is obvious that $A_d$ admits a natural $W_d$-module structure:
$$D(u,r)x^s=(u|s)x^{r+s}, \,\,\forall\,\,\,u\in \mathbb {C}^d, r,s\in \mathbb {Z}^d.$$
Using this module structure, we can form the semi-direct product Lie algebra $\widetilde{W}_d=W_d\ltimes A_d$,
called the extended Witt algebra. The intertwining Lie bracket between $W_d$ and $A_d$ is given by the
module action, that is,
$$[D(u,r), x^m]=(u|m)x^{r+m}, \,\,\forall\,\,\,u\in \mathbb {C}^d, r,m\in \mathbb {Z}^d.$$
We will denote  by $\mathcal{M}_{W_d}$ the category of $W_d$-modules, by $\mathcal{M}_{\widetilde{W}_d}$ the category of $\widetilde{W}_d$-modules.

Note that $W_d$ also has a natural module structure over the commutative associative algebra $A_d$: 
 $$ x^sD(u,r)=D(u,r+s), \ \forall \ u\in \mathbb {C}^d, r,s\in \mathbb {Z}^d.$$

\begin{definition} A  $\widetilde{W}_d$-module $P$ is  called admissible if the action of the subalgebra $A_d$ is an associative algebra action, i.e.,
$$x^{0}v=v, \  x^{n+r}v= x^r  x^nv,\ \forall\  r,n\in \Z^d, v\in P.$$
\end{definition}

An admissible $\widetilde{W}_d$-module was called a $(W_d, A_d)$-module in \cite{D}.

\subsection{$\W_d$-modules}
Let $V$ be a $\gl_d$-module and $P$ be an admissible $\widetilde{W}_d$-module.
Let $\F(P, V)=P\otimes V$. We define the  action $\W_d$ and $A_d$ on $\mathcal{F}(P, V)$ as follows:
\begin{equation}\label{2.1}
 D(u,r)(z\otimes y)= D(u,r)z\otimes y+ x^rz \otimes(ru^T) y,
\end{equation}
\begin{equation}\label{2.4}
x^r(z\otimes y)=x^rz\otimes y ,
\end{equation}
where $u\in\mathbb{C}^d$, $r\in\mathbb{Z}^d, z\in P$ and $y\in V$.
We can rewrite (\ref{2.1}) as
\begin{equation}\label{2.5}(x^r\partial_j)(z\otimes y)= (x^r\partial_j)z\otimes y+\sum_{i=1}^dr_i x^r z \otimes(E_{ij} y).\end{equation}
 
\begin{proposition}\label{p} Let $V, V_1, V_2$ be $\gl_d$-modules and $P$ be an admissible $\widetilde{W}_d$-module. Then 
\itemize
\item[(a).] $\F(P, V)$ is an admissible $\widetilde{W}_d$-module;
\item[(b).] $\F( \F(P, V_1), V_2)\cong \F(P, V_1\otimes V_2)$.
\end{proposition}

\begin{proof} (a).  We need to verify that
$$\aligned
D(u,r)D(v, s)(z\otimes y)- D(v, s)D(u,r)(z\otimes y& = D(w, r+s)(z\otimes y);\\
D(u, r) x^m(z\otimes y)-x^m D(u,r)(z\otimes y) & =(u\mid m)x^{m+r}(z\otimes y),
\endaligned$$ 
for all $u,v\in\mathbb{C}^d$, $z\in P, y\in V$, $ r,s, m\in\mathbb{Z}^d$, where $w=(u\,|\,s)v-(v\,|\,r)u$.
This is straightforward and we omit the details.

(b). We define the linear map
\begin{equation*}\begin{array}{crcl}
\varphi: & \F( \F(P, V_1), V_2)      & \to     & \F(P, V_1\otimes V_2),\\ 
         & (z\otimes y_1)\otimes y_2 & \mapsto & z\otimes (y_1\otimes y_2),
\end{array}\end{equation*}
for $z\in P$, $y_1\in V_1$, and $y_2\in V_2$. It is also straightforward to check that
$$\aligned 
\varphi\big( D(u,r)((z\otimes y_1)\otimes y_2)\big) & =D(u,r) \varphi\big((z\otimes y_1)\otimes y_2\big);\\
\varphi\big( x^r((z\otimes y_1)\otimes y_2)\big) & = x^r\varphi((z\otimes y_1)\otimes y_2).
\endaligned$$
So $\F( \F(P, V_1), V_2)\cong \F(P, V_1\otimes V_2)$ as admissible $\widetilde{W}_d$-modules.
\end{proof}

Recall that the classical Weyl  algebra ${\K}_d$ is the  unital  simple associative algebra $\C[x_1^{\pm 1},\cdots,x_d^{\pm 1},\partial_{1},\cdots,\partial_{d}]$. 
It is worthy to remark that any $\K_d$-module can be viewed as an admissible $\widetilde{W}_d$-module 
since $W_d\ltimes A_d$ can be regarded as a Lie subalgebra of $\K_d$ spanned by all 
$x^r\partial_i$ and $x^r$ with $r\in\Z^d$ and $i=1,\cdots,d$. 
Similarly as before, we denote  by $\mathcal{M}_{\K_d}$ the category of $\K_d$-modules.

In \cite{LLZ}, we have studied another module structure on $P\otimes V$ defined as follows:
\begin{equation}\label{Action2}
 (x^{r-e_j}\partial_{j})\circ  (z\otimes y)
=((x^{r-e_j}\partial_{j})z)\otimes y+ \sum_{i=1}^nr_i(x^{r-e_i}z)\otimes E_{ij}(y),
\end{equation}
and $x^r (z\otimes y)=(x^{r}z)\otimes y$
for all $r=(r_1,\cdots,r_d)^T\in \Z^d, z\in P$ and $y\in V.$
Denote this module by $F(P, V)$. 
Note that in \cite{LLZ}, the $W_d$-module $F(P,V)$ is defined only when $P$ is a $\K_d$-module, 
however, the definition is still valid for an admissible $\widetilde{W}_d$-module.

It is easy to see that, $\F(P, V)$ is a weight $\W_d$-module if   $P$ is a weight $\W_d$-module, while $F(P, V)$ is a weight $\W_d$-module if   both $P$ and $V$ are  weight modules. Next we will prove that $\F(P, V)$ and $F(P, V)$ are somewhat the same when $V$ is a weight $\gl_d$-module.

For any $\lambda=(\l_1,\cdots,\l_d)\in \C^d$, we have the automorphism $\tilde{\lambda}$ of ${ {\K}}_d$ defined by 
$$\tilde{\l}(x^{\a})=x^{\a}, \quad \tilde{\l}(\partial_j)=\partial_j- {\lambda}_j,\ \forall\ j=1,2,\ldots,d.$$
Then for any module $P$ over the associative algebra ${{\K}}_d$, we have the new module $P^{\tilde{\lambda}}=P$ with the action 
$$y \circ_{\lambda} p={\tilde{\lambda}}(y) p,\ \forall\ y\in {{\K}}_d,\ p\in P^{\tilde{\lambda}}.$$

If  $V$ is a simple weight $\gl_d$-module, there is a $\lambda\in\C^d$ such that
\begin{equation}\label{weightset}
V=\bigoplus_{\mu\in \Z^d}V_{\lambda+\mu},
\end{equation}
where $V_{\lambda+\mu}=\{  v\in V\mid E_{ii}v=(\lambda_i+\mu_i) v, \ \forall\ i=1, 2, \cdots, d\}$.
We will need the following theorem in the next sections.

\begin{theorem} If $V$ is a simple weight module over $\gl_d$ with decomposition 
(\ref{weightset}) for some $\l\in\C^d$ and $P$ is a module over the associative algebra ${ {\K}}_d$. 
Then the linear  map
$$\aligned
\Phi: \F(P, V) & \rightarrow& F(P^{\tilde{\lambda}}, V)\hskip5pt&\\
p \otimes v_\mu& \mapsto& x^{-\mu}p \otimes v_\mu,&\quad\forall\ v_\mu\in V_{\lambda+\mu},\ p\in P
\endaligned$$ 
is a $W_d$-module isomorphism.
\end{theorem}

\begin{proof}  Clearly $\Phi$ is a bijection.  Moreover,
for any $r\in\Z^d, j=1,\cdots,d$, $p\in P$ and $v_\mu\in V_{\lambda+\mu}$, we can check that
$$\aligned 
  &(x^{r-e_j}\partial_j)\circ\Phi(p \otimes v_\mu)=(x^{r-e_j}  \partial_j)\circ (x^{-\mu} p\otimes v_\mu)\\
= & x^{r-e_j} (\partial_j-\l_j) x^{-\mu} p\otimes v_\mu+\sum_{i=1}^d r_ix^{r-\mu-e_i}p\otimes E_{ij}v_\mu\\
= & x^{r-e_j-\mu}(\partial_j-\l_j-\mu_j)p\otimes v_\mu+\sum_{i=1}^d r_ix^{r-\mu-e_i}p\otimes E_{ij}v_\mu\\
= & x^{r-e_j-\mu}\partial_jp\otimes v_\mu-x^{r-e_j-\mu}p\otimes E_{jj}v_\mu+\sum_{i=1}^dr_ix^{r-\mu-e_i}p\otimes E_{ij}v_\mu\\
= & x^{r-e_j-\mu}\partial_jp\otimes v_\mu+\sum_{i=1}^dx^{r-\mu-e_i}p\otimes (r_i-\delta_{ij})E_{ij}v_\mu\\
= & \Phi \big((x^{r-e_j} \partial_j)(p \otimes v_\mu)\big),
\endaligned$$ 
where in the fourth and sixth equalities, we have used the facts
$E_{jj}v_\mu=(\lambda_j+\mu_j) v_\mu$ and $E_{ij}v_\mu\in M_{\lambda+\mu+e_i-e_j}$ respectively.
So $ \Phi$ is a $W_d$-module isomorphism.
\end{proof}

For each $\a\in \C^d$ and $a\in\C$, there is an  admissible $\widetilde{W}_d$-module structure on $A_d$ as follows:
$$D(u,r) x^n=(u\mid n+\a-ar) x^{n+r}, \ x^r x^n=x^{n+r}.$$ We denote this module  over $\widetilde{W}_d$ (or $ {W}_d$) by
$A_d(\a, a)$. Let $V$ be a simple  $\gl_d$-module on which the identity matrix acts as a
complex scalar $b$.
By (\ref{2.1}), the action of $\W_d$ on $\F(A_d(\a, a), V)=A_d(\a,a) \otimes  V$ is defined by
$$D(u,r) ( x^n\otimes v)=(u\mid \a+n-ar)  x^{n+r}\otimes  v+ x^{n+r}\otimes(ru^Tv),$$
where $u\in\C^d,r,n\in\Z^d, v\in V$. In this case,  $\F(A_d(\a,a), V)$ is a weight module over $\W_d$ which was studied by many authors, see \cite{BF, E1, E2,GZ, L1, L2, L3, L4, L5, Sh}.
It was showed that $\F(A_d(\a, 0), V)$ is  a reducible module over $W_d$
if and only if $V$ is isomorphic to the simple finite dimensional module
whose highest weight is a fundament weight $\delta_k$ and $b=k$,
where $k\in\Z$ with $1\leq k\leq d-1$, or $\dim V=1$, $\a\in\Z^d$ and $b\in\{0, d\}$, see \cite{LZ}.
Recently Billig and Futorny \cite{BF} showed that any  irreducible cuspidal $\W_d$-module is isomorphic
 to some irreducible subquotient
of $\F( A_d(\a, 0), M)$.

Since there are a lot of known simple $\gl_d$-modules $V$ and simple admissible $\widetilde{W}_d$-modules $P$, (see \cite{N1, N2, TZ2} and the references therein),  we can actually obtain a lot of new $\W_d$-modules from the above construction of $\F(P, V)$. In the next sections we will study one such class of $\W_d$-modules.

\section{The  weighting functor $\mathfrak{W}$}

We will apply the weighting functor $\mathfrak{W}$ introduced in \cite{N2} to the module categories $\mathcal{M}_{W_d}$, $\mathcal{M}_{\widetilde{W}_d}$ and $\mathcal{M}_{\K_d}$, which will be widely used in the next sections.

For $\a\in\C^d$,  let $I_\a$ be the maximal ideal of  $U(\mh)$ generated by $$\partial_1-\a_1,\dots, \partial_d-\a_d.$$
For a $\W_d$-module $P$ and  $\a\in\C^d$, set $P_{\a}:= P/I_{\a}P$. 
Denote $$\mathfrak{W}^{(\a)}(P):=\bigoplus_{n\in\Z^d}( P_{n+\a}\otimes x^n).$$
Since the module structures of $\mathfrak{W}^{(\a)}(P)$ for distinct $\a$ are similar, 
we study $\mathfrak{W}^{(0)}(P)$ in the rest of this section and denote it by $\mathfrak{W}(P)$ for short.
The general construction will be used in Section 5.
By Proposition 8 in \cite{N2}, we have the following construction.
\begin{proposition} The vector space $\mathfrak{W}(P)$  becomes a  ${\W}_d$-module   under the following action:
\begin{equation}\label{3.3}
D(u,r)\cdot((v+I_{n}P)\otimes x^n):= (D(u,r)v+I_{n+r}P)\otimes x^{n+r}.
\end{equation}
Moreover, if $P$ is an admissible $\widetilde{W}_d$-module, then  $\mathfrak{W}(P)$ becomes an  admissible $\widetilde{W}_d$-module  via
\begin{equation}\label{2.4} 
x^r\cdot((v+I_{n}P)\otimes x^{n}):= (x^rv+I_{n+r}P)\otimes x^{n+r},
\end{equation}
where $ u\in\C^d,n, r\in\Z^d$.
\end{proposition}

In many cases, the  $\W_d$-module $\mathfrak{W}(P)$  is $0$. For example, if $P$ is a simple  weight $W_d$-module with a weight not in $\Z^d$, one can easily see that $\mathfrak{W}(P)=0$. We also note that $\mathfrak{W}(P)=P$ if
 $P$ is a simple  weight $W_d$-module with a weight  in $\Z^d$.

\begin{remark} The $\W_d$-module $\mathfrak{W}(P)$  is always a weight module since
$$D(u,0)\cdot((v+I_{n}P)\otimes x^n)=(u\mid n)(v+I_{n}P)\otimes x^{n},$$
for all $v\in P$, and
$P_{n}\otimes x^n$ is a weight space for each $n\in\Z^d$.
\end{remark}

The action of the weighting functor $\mathfrak{W}$ on a $W_d$-module homomorphism $f:P_1\to P_2$ is as follows
$$ \aligned 
\mathfrak{W}(f):\quad  \mathfrak{W}(P_1) & \to       \mathfrak{W}(P_2),\\
                  v+I_nP_1        & \mapsto   f(v)+I_nP_2,\ \forall\ v\in P_1.
\endaligned$$
The following properties are easy to verify

\begin{lemma}\label{morphism} 
Let $P_1, P_2\in  \mathcal{M}_{W_d}$ and $f:P_1\to P_2$ be a $W_d$-module homomorphism. 
\begin{itemize}
\item[(a).] If $f$ is onto, so is  $\mathfrak{W}(f)$; 
\item[(b).] If $f$ is an isomorphism, so is  $\mathfrak{W}(f)$.
\end{itemize}
\end{lemma}

\begin{remark} Note however, we do not have similar result for monomorphisms,
that is, injectivity of $f$ does not necessarily implies the injectivity of $\mathfrak{W}(f)$.
\end{remark}

\begin{proposition}\label{p3.5} 
For any admissible $\widetilde{W}_d$-module $P$ and $\gl_d$-module $M$, 
we have that $\mathfrak{W}(\F(P, V))\cong \F( \mathfrak{W}(P), V)$ as admissible $\widetilde{W}_d$-modules.
\end{proposition}

\begin{proof} For $n\in\C^d$,  we have
$$\aligned 
\F(P, V)_n= & \F(P, V)/I_n\F(P, V)=(P\otimes V)/(I_n P)\otimes V\\ 
          = & (P/I_n P)\otimes V=P_n\otimes V.
\endaligned$$ 
So $\mathfrak{W}(\F(P, V))=\bigoplus_{n\in\Z^d}(P_n\otimes V)\otimes x^n$.
Consider the linear map $\varphi: \mathfrak{W}(\F(P, V))\to \F( \mathfrak{W}(P), V)$ defined by
$$\varphi((y_n\otimes v)\otimes x^n)=(y_n\otimes x^n)\otimes v$$
for any $n\in\Z^d, y_n\in P_n, v\in V$. One can easily verify that $\varphi$ is an admissible $\widetilde{W}_d$-module isomorphism.
\end{proof}

Let us first recall the non-weight $\W_d$-modules
$\Omega(\lambda, a)$ from \cite{TZ1} for any $a\in\C$ and $\lambda=(\l_1,\l_2,\dots,\l_d)^T\in (\C^\star)^d$. 
Denote by $\Omega(\lambda,a) =\C[t_1,\dots,t_d]$, 
the polynomial associative algebra over $\C$ in the commuting indeterminates $t_1,\dots,t_d$. 
For simplicity, if $f(t_1,t_2,...,t_d)\in \C[t_1,\dots,t_d]$,
$r=(r_1,\cdots, r_d)^T\in \C^d$ and $u\in\C^d$, denote
$$\aligned
f(t-r):&= f(t_1-r_1,\cdots,t_d-r_d),\\
(u\mid t+r):&=\sum_{i=1}^d u_i(t_i+r_i).
\endaligned$$

The action of $\W_d$ and $A_d$  on  $\Omega(\lambda,a)$ is defined as follows:
\begin{equation}\label{2.3}\begin{split}
D(u,r) \cdot f(t)& =\lambda^r(u\mid t-ar)f(t-r)\\
x^r\cdot f(t)& =\lambda^rf(t-r),
 \end{split}\end{equation}where $ u\in\C^d,r\in\Z^d$ and $f(t)\in \C[t_1,\dots,t_d]$. It is easy to see that $\Omega(\lambda,a )$ is an admissible $\widetilde{W}_d$-module. It is actually a $\K_d$-module if $a=1$.

 \begin{example} Consider the weight  $\widetilde{W}_d$-module
   $\mathfrak{W}(\Omega(\lambda,a ))$ for any  $a\in\C$ and $\lambda=(\l_1,\dots,\l_d)\in (\C^\star)^d$. Since
   $$\aligned D(u, r)(1+I_n\Omega(\lambda,a ))=&\lambda^r(u\mid t-ar)+I_{n+r}\Omega(\lambda,a )
   \\
   =&\lambda^r(u\mid n+r-ar)+I_{n+r}\Omega(\lambda,a ),\endaligned$$
   where $u\in\C^d, r\in\Z^d$, we see that
   $\mathfrak{W}(\Omega(\lambda,a ))$ is  isomorphic to the module $A_d(0,a-1 )$.
 \end{example}

\section{Simple admissible $\widetilde{W}_d$-modules that are finitely generated   $U(\mh)$-modules}

In this section we determine the category $\widetilde{\mathcal{H}}$ consisting of admissible $\widetilde{W}_d$-modules  that are  finitely generated $U(\mh)$-modules. 

Let  $V$ be a finite dimensional $\gl_d$-module, and $P$ a $\K_d$-module that is a finitely generated  $U(\mh)$-module. Then we have the admissible $\widetilde{W}_d$-module $\F(P, V)$ that is a  finitely generated  $U(\mh)$-module.  
We will show in this section that simple modules in
$\widetilde{\mathcal{H}}$ are exactly the  $\widetilde{W}_d$-modules $\F(P, V)$ for simple  $\K_d$-modules $P$  that is a finitely generated $U(\mh)$-module, and finite dimensional simple $\gl_d$-modules $V$.

Fix any nonzero $M\in\widetilde{\mathcal{H}}$. 
Denote  $M'=\{v\in M|{\rm ann}_{U(\mh)}(v)\ne 0\}$. It is easy to see that   $M'$ is a $\widetilde{W}_d$-module.

\begin{lemma}\label{torsion} 
The admissible $\widetilde{W}_d$-module  $M$ is torsion-free over $U(\mh)$. 
\end{lemma}

\begin{proof} Since $U(\mh)=\C[\partial_1,\cdots,\partial_d]$ is Noetherian, then $M'$ is a finitely generated   $U(\mh)$-module. We see that  $0\ne I= {\rm ann}_{U(\mh)}(M')$ is an idea of $U(\mh)$. Clearly, for any $\alpha\in\Z^d$, we have  $x^{\a}M'\subset  M'$, and the linear maps 
$$x^\a:M'\to M', \,\,\, x^{-\alpha}: M'\to M'$$ 
are inverse of each other.
We deduce that  $M'=x^{\a}M'$ for all $\a\in \Z^d$. We see that, $f(\partial)=f(\partial_1,\ldots,\partial_d)\in I$ if and only if $f(\partial-\a)\in I$, for all $\a\in \Z^d$, which implies that $I=U(\mh)$, i.e. $M'=0$. Hence $M$ is $U(\mh)$-torsion free. 
\end{proof}

Now we see that any nonzero modules in $\widetilde{\mathcal{H}}$ are 
finitely generated and torsion free over $U(\mh)$. 
Before continuing, we give a simple property of such modules.

\begin{lemma}\label{cap I_iP} 
Let $P$ be finitely generated torsion free module over $U(\mh)$, and $\{I_i, i\in S\}$ 
is a family of ideals of $U(\mh)$. If  $\bigcap_{i\in S} I_i=0$ then $\bigcap_{i\in S} I_iP=0$.
\end{lemma}

\begin{proof}
Since $U(\mh)$ is an integral domain, by a well-known result in Commutative Algebra, 
$P$ can be imbedded in a free $U(\mh)$-module $N$ of finite rank. The lemma follows easily from the fact $\bigcap_{i\in S} I_iN=0$.
\end{proof}

We go on to determine the structures of simple modules in $\widetilde{H}$. 
First we recall the Lie algebra $\T$ from \cite{E2}. 
Let $U$ be the universal enveloping algebra of $\widetilde{W}_d$. Let $\mathcal{I}$ be the two
sided ideal of $U$ generated by 
$$ x^0-1, x^r\cdot x^s-x^{r+s},\ \forall\ r,s\in \Z^d.$$ 
Let $\overline{U}=U/\mathcal{I}.$
Note that an admissible $\widetilde{W}_d$-module is just a $U$-module annihilated by $\mathcal{I}$,
or, just a $\overline{U}$-module.
In particular, $M$ is a $\overline{U}$-module.

Denote $T(u,r)= x^{-r}D(u,r)-D(u,0)+\mathcal{I}\in\overline{U}$. We will identify element in $\widetilde{W}_d$ with its image in $\overline{U}$.  It is easy to verify that
$$\aligned 
&[T(v, s),T(u, r)]\hskip -4pt =\hskip -4pt (u|s)T(v, s)\hskip -4pt -\hskip -4pt (v|  r)T(u,r)\hskip -4pt+\hskip -4ptT((v|r)u-(u|s)v, r+s),\\
&[D(v, 0),T(u, r)]=[x^s,T(u, r)]=0,\ \forall\ u,v\in\C^d, r,s\in\Z^d.
\endaligned$$
We use $\T$ to denote the Lie subalgebra of $\overline{U}$ generated by
operators $T(u, r)$, $u\in\C^d, r\in\Z^d$. Let $J$ denote the subspace  of $\T$ spanned by 
$$T(u; r, m)=T(u,r+m)-T(u,r)-T(u,m),\ \forall\ u\in\C^d, r,m\in\Z^d.$$
It is straightforward to check (or, cf. \cite{E2}) that $J$ is an ideal of $\T$ and  the linear map
\begin{equation}\label{iso-gl_d}
\T/J\to \gl_d(\C); \ T(e_i, e_j)\mapsto E_{j,i},\ \forall\ i,j\in\{1,2,\cdots,d\}
\end{equation}
is a Lie algebra isomorphism. So $z_d=\sum_{i=1}^d T(e_i,e_i)\in Z(\T/J)$. 
We can check that $[z_d, \widetilde{W}_d]=0$ and hence $z_d$ is an endomorphism of the $\widetilde W_d$-module $M$. 

We denote the fraction field of $U(\mh)$ by $H$. Since $M$ is a finitely generated torsion-free $U(\mh)$-module, we have the nonzero finite-dimensional vector space $M_H:=H\ot_{U(\mh)}M$ over $H$. 
Define the $U(\mh)$-rank of $M$ as rank$(M)=\dim_H M_H$.
If $\mg$ is a Lie algebra over $\C$, then the tensor product
$\mg_H:=H\otimes_{\C}\mg$ can be considered as a Lie algebra over $H$ by defining $\k(\k_1\ot g)=(\k\k_1\ot g)$ for $\k,\k_1\in H, g\in \mg$.

From now on in this section, we will assume that $M\in\widetilde{\mathcal{H}}$ is a simple $\widetilde{W}_d$-module. 
Thus $z_d$ acts as a scalar on $M$.

\begin{lemma}\label{T/J} 
We have $JM=0$, and further $M$ is  a  module over $\T/J$. 
\end{lemma}

\begin{proof} By definition, $M$ is a  $\T$-module. Recall that $\mathfrak{W}(M)$ is a cuspidal admissible $\widetilde{W}_d$-module. Using Theorem 2.9 in \cite{E2} we see that    $J^k \mathfrak{W}(M)=0$ for some positive integer $k$, i.e., $J^k M\subseteq I_nM$ for all $n\in \Z^d$. Since $M$ is a torsion free $U(\mh)$-module of finite rank, we have $\bigcap_{n\in \Z^d} I_nM=0$ by Lemma \ref{cap I_iP}. Thus $J^kM=0$. Hence $JM\ne M$.
In the algebra $\overline{U}$, we have
$$\aligned 
& [D(v, s),T(u, r)] \equiv x^s\big((v|r)T(u, s)\hskip -4pt -\hskip -4pt (u|\hskip -2pt s)T(v,r)\big)\quad \mod x^sJ,\\
& [D(v, s),T(u; m, r)] \equiv - x^s(u|\hskip -2pt s)T(v;m,r)\quad \mod x^sJ,
\endaligned$$
for all $u,v\in\C^d$ and $ m, r,s\in\Z^d$.
So  $JM$ is clearly a $\widetilde W_d$-submodule of $M$. Since $M$ is a simple  $\widetilde W_d$-module, we have $JM=0$. Thus $M$ is a  module over $\T/J$.
\end{proof}

\begin{lemma} \label{V} The $W_d$-module $M$ has a finite dimensional simple submodule over $\T/J\cong\gl_d$.

\end{lemma}

\begin{proof}
We can  consider $\T$  as a subalgebra of $\End_{U(\mh)}(M)$.
Over the field $H$, $M_H$ is a finite dimensional module over the Lie algebra $\T_H$.
By Lemma \ref{T/J}, we see that $JM=0$ and $J_HM_H=0$.
Hence $M_H$ can be viewed as a finite dimensional $(\T_H/J_H)$-module.
Noticing $\T/J\cong\gl_d(\C)$, we can easily see that $\T_H/J_H\cong\gl_d(H)$ which is a  finite dimensional split reductive Lie algebra over $H$. 
Recall that $z_d$, which corresponds to the identity matrix in $\gl_d(H)$, acts as a scalar on $M$.
By Theorem 8 on Page 79 in \cite{J}, we know that the $\gl_d(H)$-module $M_H$ is completely reducible. Let $V_1$ be a simple  $\gl_d(H)$-submodule of $M_H$.
By Theorem 3 on Page 215 in \cite{J}, we know that $V_1$ is a highest weight module of a dominant weight. Let $v$ be a highest weigh vector of $V_1$. We may assume that $v\in M$. Let $V$ be the $(\T/J)$-submodule of $M$ generated by $v$.
Then $V$ is a finite dimensional irreducible submodule of $M$ over   $\T/J\cong\gl_d(\C)$.
\end{proof}

\begin{lemma} We have the   associative algebra isomorphism 
\begin{equation}\label{iso-iota}
\iota:\K_d\otimes U(\T)\rightarrow \overline{U}, \  \  
\iota(x^r \partial^{\a} \otimes y)=x^r \cdot \prod_{j=1}^d D(\varepsilon_j,0)^{\a_j}\cdot y+I,
\end{equation}
where $r, \a\in\Z^d, y\in U(\T)$.
\end{lemma}

\begin{proof} Note that $U(\T)$ is an associative subalgebra of $\overline{U}$. Since the restrictions of $\iota$ on $\K_d$ and $U(\T)$ are well-defined and   homomorphisms of associative algebras, so $\iota$ is well-defined and is an homomorphism of associative algebras. From the definition of $T(u, r)\in T$ for $u\in\C^d, r\in\Z^d$, we have
$D(u,r) =x^{r}(T(u,r)+D(u,0))$, i.e.,
$$\iota(x^{r}\otimes T(u,r)+x^rD(u,0)\otimes 1)=D(u, r),$$
 we can   see that    $\iota$ is  an epimorphism.

 By PBW Theorem we know that $\overline{U}$ has a basis consisting monomials in the variables
 $$D(e_i, r): r\in\Z^d\setminus\{0\}, i\in\{1,2,\cdots, d\}$$
 over $\K_d$. Therefore  $\overline{U}$ has a basis consisting monomials in the variables
 $$T(e_i, r): r\in\Z^d\setminus\{0\}, i\in\{1,2,\cdots, d\}$$
 over $\K_d$. So $\iota$ is injective and hence an isomorphism.
 \end{proof}

Now any (simple) $\overline{U}$-module can be also considered as a (simple) module over $\K_d\otimes U(\T)$ via the isomorphism $\iota$.
The following result is well-known.

\begin{lemma}\label{V'}Let $A, B$ be unital associative algebras and $B$ have a countable basis. \begin{itemize}\item[(a).] If $M$ is  a simple module   over $A\otimes B$ that  contains a simple $B=\C\otimes B$ submodule $V$, then $M\cong W\otimes V$ for  a simple $A$-module $W$. \item[(b).] If $W$ and $V$ are simple modules over $A$ and $B$ respectively, then $W\otimes V$ is a simple  module   over $A\otimes B$.\end{itemize}\end{lemma}

Now we can determine all simple modules in $\widetilde{\mathcal{H}}$.

\begin{lemma}\label{Simple} Let  $V$ be a   simple   $\gl_d$-module, and $P$ a simple $\K_d$-module. Then the admissible $\widetilde{W}_d$-module $\F(P, V)$ is simple.
\end{lemma}

\begin{proof}  Regard $\F(P, V)$ as $\K_d\otimes U(\T)$ module by $\iota$. It is clear from Lemma \ref{V'} that $\F(P,V)$ is a simple $\K_d\otimes U(\T)$  module.
Hence it is also a simple  admissible $\widetilde{W}_d$-module.
\end{proof}

\begin{theorem} \label{main1}Let $M$ be a simple admissible $\widetilde{W}_d$-module   that  is a finitely generated  $U(\mh)$-module  when restricted to $U(\mh)$.
Then  $M\cong \F(P, V)$ for a simple  $\K_d$-module $P$  that is a finitely generated  (torsion-free) $U(\mh)$-module, and  a finite dimensional simple $\gl_d$-module $V$.
\end{theorem}

\begin{proof} By Lemmas \ref{V} and \ref{V'}, there is  a simple  $\K_d$-module $P$ and a finite dimensional simple $\gl_d$-module $V$ so that $M\cong P\otimes V$ as $\overline{U}$-modules, 
via the Lie algebra isomorphism in \eqref{iso-gl_d} and associative algebra isomorphism in \eqref{iso-iota}. 
More precisely, we can deduce the action of $\widetilde{W}_d$ on $P\otimes V$:
$$\aligned 
D(u,r)(y\otimes v)=&x^{r}(T(u,r)+D(u,0))(y\otimes v)\\
                  =&(x^{r}D(u,0)y)\otimes v+(x^ry)\otimes (ru^T)v, 
\endaligned$$
for all $u\in\C^d, r\in\Z^d,  y\in P, v\in V$, coinciding with 
the definition of the $\widetilde{W}_d$-module $\F(P, V)$.
Hence $M\cong \F(P,V)$.

At last, note that
$$f(\partial)(y\otimes v)=f(\partial)y\otimes v, \quad \forall\ f(\partial)\in U(\mh), y\in P, v\in V.$$ 
Since $M$ is a finitely generated torsion-free  $U(\mh)$-module, so is $P$. 
\end{proof}

\begin{corollary}\label{cor4.8} Let $M$ be a simple admissible $\widetilde{W}_d$-module  that  is a free $U(\mh)$-module of finite rank. Then  $M\cong \F(P, V)$ for a simple $\K_d$-module $P$ that is a free $U(\mh)$-module of finite rank, and  a finite dimensional simple $\gl_d$-module $V$.\end{corollary}

\begin{proof}From the proof of the above theorem, we only need to show the $P$ is also a free $U(\mh)$-module. In fact we have $M\cong P^m$ as $U(\mh)$ modules where $m=\dim V$. Then $P$ is a finitely generated projective module over $U(\mh)$, which from Quillen-Suslin Theorem is free.\end{proof}

\begin{corollary}Any admissible $\widetilde{W}_d$-module that is a finitely generated  $U(\mh)$-module has a finite composition length as $\widetilde{W}_d$-module.\end{corollary}

\begin{proof} Note that any admissible $\widetilde{W}_d$-module $M$ that is a finitely generated   $U(\mh)$-module is $U(\mh)$ torsion-free. Since $U(\mh)$ is a Noetherian integral domain, any $U(\mh)$-submodules of $M$ is finitely generated as $U(\mh)$-module. Hence the length of any composition series of $M$ cannot exceed the rank of $M$ as $U(\mh)$-module. Thus $M$ has a finite composition length as $\widetilde{W}_d$-module.\end{proof}

\section{Simple $\W_d$-modules   that are  finitely generated   $U(\mh)$-modules}

In this section we study the category ${\mathcal{H}}$ consisting of $\W_d$-modules that are finitely generated   $U(\mh)$-modules.  Let $\mathcal{W}$ be the category  consisting of weight $\W_d$-modules with finite dimensional weight spaces.
Corollary \ref{p3.5} tells us that there is a close link between these two categories. 
We will generalize  the concept of the covering module established in \cite{BF} to the category ${\mathcal{H}}$.

In this section, we always fix a nontrivial $W_d$-module $M\in\mathcal{H}$.

\begin{lemma}\label{torsion-1} 
Any nontrivial simple $W_d$-module $M\in\mathcal{H}$ is torsion-free over $U(\mh)$.
\end{lemma}

\begin{proof} Recall the $W_d$-submodule  $M'=\{v\in M|{\rm ann}_{U(\mh)}(v)\ne 0\}$.  Suppose that $M$ is not torsion-free, i.e., $M'\neq0$. Then $M= M'$ and $J={\rm ann}_{U(\mh)}(M)\ne 0$ since $M$ is finitely generated over $U(\mh)$. 
Recall that $I_\a$ is the maximal ideal of $U(\mh)$ generated by $\partial_i-\a_i, i=1,\cdots,d$.

We claim that $I_\a M=M$ for all $\a\ne 0$. Otherwise, say $I_{\a_0}M\neq M$ for some $\a_0\ne 0$. 
Consider the the Harish-Chandra $W_d$-module defined at the beginning of Section 3, i.e., 
$\mathfrak{W}^{(\a_0)}(M)=\bigoplus_{n\in \Z^d} (M/I_{n+\a_0}M)\otimes x^n$.
From the irreducible Harish-Chandra module theory (cf. \cite{BF}), we know that 
$M/I_{n+\a_0}M\ne 0$ for all $n+\a_0\neq0$. Hence $J\subseteq I_{n+\a_0}$ for all $n+\a_0\neq0$, 
for otherwise $I_{n+\a_0}M=(I_{n+\a_0}+J)M=U(\mh)M=M$.
We have $J\subseteq \bigcap_{n\in\Z^d\setminus\{-\a_0\}} I_{n+\a_0}=0$, a contradiction.

Since $M$ is a finitely generated $U(\mh)$-module, it has a maximal $U(\mh)$-submodule $K$. 
Then we have $M/K\cong U(\mh)/I_{\a_1}$ for some $\a_1\in\C^d$ and $I_{\a_1}M\subseteq K\ne M$, 
forcing $\a_1=0$. So we have proved that $I_0M \neq M$. Note that $D(u, r)I_\alpha=I_{\alpha+r}D(u, r)$. Then $\bigcap_{\a\in \C^d}I_{\a}M=I_0M$ becomes a proper nonzero $W_d$-submodule of $M$, a contradiction. Thus $M$ is torsion-free over $U(\mh)$.
 \end{proof}

Now consider $W_d$ as the adjoint $W_d$-module. We can make the tensor product $\W_d$-module $W_d\otimes M$ into an admissible $\widetilde{W}_d$-module by defining
$$x^s(D(u, r)\otimes y)=D(u, r+s)\otimes y,\ \forall\ u\in\C^d, r, s\in\Z^d.$$
For  $u\in \C^d, r\in\Z^d, y\in M$, we define $D(u, r)\boxtimes y\in {\text{Hom}}_{\C}(A_d,M)$ as
$$(D(u, r)\boxtimes y)(x^s)=D(u, r+s) y,\ \forall\ s\in\Z^d.$$
Denote
$$W_d\boxtimes M=\span\{w\boxtimes y\ |\ w\in W_d, y\in M\}\subset {\text{Hom}}_{\C}(A_d,M).$$

We define  the canonical linear map
$$\psi: W_d\otimes M\to W_d\boxtimes M; \quad w\otimes y\mapsto w\boxtimes y.$$
It is easy to see that the kernel of $\psi$ is a $\widetilde{W}_d$-submodule of the tensor module $W_d\otimes M$.
Thus we can make $W_d\boxtimes M$ into an admissible $\widetilde{W}_d$-module vis $\psi$, which is isomorphic to $(W_d\otimes M)/\ker\psi$.
As in \cite{BF}, we call this admissible $\widetilde{W}_d$-module as {\it the cover of} $M$. 
The action of $\widetilde{W}_d$ on $W_d\boxtimes M$ can be written explicitly as follows:
$$\aligned
& D(u,r)(D(v,s)\boxtimes y)=[D(u,r),(D(v,s)]\boxtimes y+D(v,s)\boxtimes D(u,r)y\\
& x^r(D(v,s)\boxtimes y)=D(v,r+s)\boxtimes y,\ \forall\ u,v\in\C^d, r,s\in\Z^d, y\in M.
\endaligned$$

The following result is easy to verify.

\begin{lemma}\label{cover}  
The linear map
$$\aligned
\pi: \,\,&W_d\boxtimes M & \to\ \  & M ,\\
         &\hskip5pt  w\boxtimes y & \mapsto\ \  & wy,\quad \forall\ w\in W_d, y\in M
\endaligned$$ 
is a $W_d$-module epimorphism.
\end{lemma}

The following result gives an interesting property on commutativity of the weighting functor 
$\mathfrak{W}$ and   $W_d\boxtimes$.
 
\begin{lemma}\label{iso} 
As admissible $\widetilde{W}_d$-modules, $W_d\boxtimes {\mathfrak{W}(M)}\cong \mathfrak{W}(W_d\boxtimes {M})$.
\end{lemma}

\begin{proof} Let us first compute $I_n(W_d\boxtimes {M})$. 
For any $u, v\in\C^d, n, s\in\Z^d$ and $y\in M$, we have
$$\aligned 
  &(D(u, 0)-(u|n))(D(v,s)\boxtimes y)
= D(v,s)\boxtimes (D(u, 0)-(u|n-s))y;\endaligned$$
and $$\aligned 
  I_n(W_d\boxtimes {M})
=& \sum_{s\in\Z^d}W_d(s)\boxtimes I_{n-s}(M),
\endaligned$$
where $W_d(s)=\sum_{v\in\C^d}D(v,s)$ is the  root space of $W_d$.

Define the linear map $\gamma: W_d\boxtimes {\mathfrak{W}(M)}\rightarrow \mathfrak{W}(W_d\boxtimes {M})$ by
$$\aligned 
 & \gamma(D(u,r)\boxtimes(y+I_nM))\\
=& D(u,r)\boxtimes y +\sum_{s\in\Z^d}W_d(s)\boxtimes I_{n+r-s}(M).
\endaligned$$
The linearity and compatibility of $\gamma$ with the action of $\widetilde{W}_d$ can be verified straightforward. The details are left to the readers. We only prove that $\gamma$ is is well-defined and bijective.

Indeed, given $n\in\Z^d$ and finitely many $u_i\in\C^d, r_i\in\Z^d, y_i\in M$,
we have $\sum\limits_{i}D(u_i,r_i)\boxtimes(y_i+I_{n-r_i}M)=0$ in $W_d\boxtimes {\mathfrak{W}(M)}$ 
if and only if  
$$\sum_{i}D(u_i,r_i+\a) y_i\in I_{n+\a}M,\ \forall\ \a\in\Z^d,$$
if and only if
$$\sum_{i}D(u_i,r_i+\a)  y_i\in  I_{n+\a}\sum_{s\in\Z^d}W_d(s+\a)(M),\ \forall\ \a\in\Z^d,$$
if and only if
$$\sum_{i}D(u_i,r_i+\a)  y_i\in \sum_{s\in\Z^d}W_d(s+\a)( I_{n-s}(M),\ \forall\ \a\in\Z^d,$$
if and only if
 $$\sum_{i}D(u_i,r_i)\boxtimes y_i+\sum_{s\in\Z^d}W_d(s)\boxtimes I_{n-s}(M)=0,$$
in $\mathfrak{W}(W_d\boxtimes {M})$. The lemma follows.
\end{proof}

\begin{lemma}\label{fin gen} Let $M\in\mathcal{H}$, which is finitely generated and torsion free over $U(\mh)$.
The admissible $\widetilde{W}_d$-module $W_d\boxtimes {M}$  is a  finitely generated  $U(\mh)$-module when restricted to  $U(\mh)$.
\end{lemma}

\begin{proof}  Since $M$  is a  finitely generated torsion-free  $U(\mh)$-module, we know that $\mathfrak{W}(M)$ is a Harish-Chandra $W_d$-module.  From Theorem 4.11 in \cite{BF} we know that  $\mathfrak{W}(M)$ is a quotient of some cuspidal admissible $\widetilde{W}_d$-module. Therefore there exists $m\in \N$ such that
$\Omega_{\a,\beta,\gamma,i,j}^{(m,\gamma)}\mathfrak{W}(M)=0$ 
for all $\a,\beta,\gamma\in \Z^d, i,j=1,2\ldots,d,$
where 
$$\Omega_{i,j,\a,\beta,\gamma}^{(m)}=\sum_{s=0}^m (-1)^s{m\choose s}D(e_i,\a-s\gamma)D(e_j,\beta+s\gamma).$$
By the module action of $W_d$ on $\mathfrak{W}(M)$, we deduce  
$$\Omega_{i,j,\a,\beta,\gamma}^{(m)}M\subseteq \bigcap_{\xi\in \Z^d} I_{\xi}M=0,\ \forall\  \a,\beta,\gamma\in \Z^d, i,j=1,2\ldots,d.$$

Let $\|\a||= |\a_1| + |\a_2|+\cdots+ |\a_d|$ for $\a=(\a_1,\ldots,\a_d)\in \Z^d$.
Note that $M=W_dM$. It is enough to prove by induction on $\|\a\|$ that 
$$D(e_i,\a)\boxtimes ( D(e_j,\beta) v)\in \sum_{\|\gamma\|\le md}  D(e_i,\gamma)\boxtimes M,$$ 
for all $v\in M, \a,\beta\in\Z^d, i,j\in\{1,\cdots,d\}.$
This is obvious for $\a\in\Z^d$ with $\|\a\|\le md$. Now we assume that $\|\a\|>md$. Without lose of generality, we may assume that $\a_1>m$. 
For any $\eta\in\Z^d$, we have
$$\aligned 
 &\sum_{s=0}^m (-1)^s {m\choose s} \Big(D(e_i,\a-se_1)\boxtimes \big(D(e_j,\beta+se_1)v\big)\Big)(x^\eta)\\
=&\sum_{s=0}^m (-1)^s   {m\choose s}D(e_i,\eta+\a-se_1)\big(D(e_j,\beta+se_1)v\big)\\
=&\Omega_{i,j,\eta+\a,\beta,e_1}^{(m)}(v)=0,
\endaligned$$
which imples 
$$\sum_{s=0}^m (-1)^s   {m\choose s} D(e_i,\a-se_1)\boxtimes (D(e_j,\beta+se_1)v)=0$$
in $W_d\boxtimes {M}$, that is,  
$$ \aligned 
& D(e_i,\a)\boxtimes (D(e_j,\beta)v)\\ =&-\sum_{s=1}^m(-1)^{s} {m\choose s} D(e_i,\a-se_1)\boxtimes\big(D(e_j,\beta+se_1)v\big), 
\endaligned$$ 
which belongs to $ \sum_{\|\gamma\|\le md}D(e_i,\gamma)\boxtimes M$
by induction hypothesis.
\end{proof}

\begin{theorem}\label{thm5.5} Let $M$ be a simple  $W_d$-module that is a finitely generated  $U(\mh)$-module  when restricted to $U(\mh)$.
Then $M$ is a simple quotient of the $W_d$-module $\F( P, V)$ for a simple $\K_d$-module $P$ that is   a finitely generated (torsion-free) $U(\mh)$-module, and a finite dimensional simple $\gl_d$-module $V$.
\end{theorem}

\begin{proof} From Lemmas \ref{cover}, \ref{fin gen},  there is an admissible  $\widetilde{W}_d$-module  $P_1$ that is  a finitely generated torsion-free $U(\mh)$-module of finite rank when restricted to $U(\mh)$ and 
a $\W_d$-module epimorphism $\phi: P_1\rightarrow M$. We may choose $P_1$ so that its rank $s$ over $U(\mh)$ is minimal.

We claim that $P_1$ is a simple  $\widetilde{W}_d$-module. Otherwise, $P_1$ has a nonzero maximal $\widetilde{W}_d$-submodule $P_2$.
Then from Lemma \ref{torsion}, $P_2$ and $P_1/P_2$ are torsion-free of rank both less than $s$. However since $M$ is simple as $W_d$ module, we must have either $\phi(P_2)=M$ or $\phi(P_2)=0$. Then either $P_2$ or $P_1/P_2$ has a simple $W_d$-quotient isomorphic to $M$. This contradicts the choice of $P_1$.

From  Theorem \ref{Simple}, we know that $P_1\cong \F( P, V)$ for a simple  $\K_d$-module $P$ that is   a finitely generated  torsion-free $U(\mh)$-module, and a finite dimensional irreducible $\gl_d$-module $V$. \end{proof}

Next we will determine all possible simple quotient modules of $\F(P, V)$ in Theorem 5.5.
Let  $V$ be a finite dimensional simple  $\gl_d$-module, and $P$ a simple $\K_d$-module   that is a finitely generated  $U(\mh)$-module. If $V$ is not isomorphic to any $\gl_d$-modules $V(\delta_k, k)$ for any $k= 1,2, \cdots, d$, from Corollary 3.6 in \cite{LLZ}  we know that the $W_d$-modules $\F(P, V)=P\otimes V$ are simple.  Now we will  determine all simple quotients of $\F(P, V(\delta_k, k))$ over $W_d$ for any $k= 1,2, \cdots, d$.

Let us first establish some general results on finitely generated torsion-free $U(\mh)$-modules.
For convenience, we denote $R=U(\mh)$, and for any prime ideal $p$ of $R$,
let $R_p$ be the localization of $R$ at $p$ and 
$P_{p}$ be the localization of the $R$-module $P$ at $p$. 
In particular, for any maximal ideal $I_\a$ of $R$ and an $R$-module $M$,
we have $M/I_\a M$ is a vector space over $R/I_\a R\cong \C$. 
Moreover, we have the following canonical isomorphisms
$$M/I_{\a}M\cong (M/I_{\a}M)_{I_{\a}}\cong M_{I_{\a}}/I_{\a}M_{I_{\a}}$$
of $R$-modules. 
Recall that $H=R_{(0)}$ is the the quotient field of $R$, where $(0)$ is the zero ideal of $R$.  

\begin{lemma}\label{L-free}Let $M$ be a finitely generated torsion-free $U(\mh)$-module. Then $M$ is a free $U(\mh)$-module if and only if $\dim M/I_{\a}M=\rank M$ for all $\a\in \C^d$.\end{lemma}

\begin{proof} We need only to prove the ``if'' part of the lemma since the ``only if'' part is clear. Let $r=\rank M$ which is actually $\dim_HM_H$. Since $M_{I_{\a}}/I_{\a}M_{I_{\a}}\cong M/I_{\a}M$ is of dimension $r$, from the Nakayama's lemma, we know that as an $R_{I_{\a}}$ module, $M_{I_{\a}}$ is generated by $r$ elements. 
Say $M_{I_{\a}}=R_{I_\a}w_1(\a)+\cdots+R_{I_\a}w_r(\a)$, where $w_1(\a),\cdots,w_r(\a)\in  M_{I_{\a}}$   are dependent on $\a$. 
Noticing that $R_{I_\a}\subset R_{(0)}=H$, we have $M_{(0)}=Hw_1+\cdots+Hw_r$. 
Since $\rank(M)=r$, we see that $w_1,\ldots,w_r$ are $H$-linearly independent, 
hence $R_{I_\a}$-linearly independent, and form an $R_{I_\a}$-basis of $M_{I_{\a}}$.

Recall that any finitely generated  module  over a Noetherian algebra is a finitely presented module. From Corollary 3.4 on Page 19 in \cite{L}, we know that $M$ is a finitely generated projective module over $R$. Hence from  Quillen-Suslin Theorem $M$ is $R$-free.\end{proof}

\begin{lemma}\label{lemma-dim}Let $M$ be a $W_d$-module that is a finitely generated torsion-free  $U(\mh)$-module of rank $r$. Then   $\dim M/I_{\a}M=r$ for all $\a\in \C^d\backslash \{0\}$. \end{lemma}

\begin{proof}
Let $v_1,v_2,\ldots, v_r\in M$ be an $H$-basis of $M_{(0)}$. Then there exists some $f\in R$ such that \begin{equation}\label{eq2}M\subseteq \frac{1}{f}(Rv_1+Rv_2+\ldots Rv_r).\end{equation}
 
 For any $\a\in \C^d\backslash\{0\}$, from the fact that $\bigcap_{n\in \Z^d\backslash\{-\a\}} I_{\a+n}=0$, we see that there exists some $n\in \Z^d\backslash\{-\a\} $ such that $f\not\in I_{\a+n}$. From (\ref{eq2}), we see that 
 $M_{I_{\a+n}}=fM_{I_{\a+n}}=R_{I_{\a+n}}v_1+R_{I_{\a+n}}v_2+\ldots R_{I_{\a+n}}v_r$,  and further 
 $\{v_1,\ldots,v_r\}$ becomes an $R_{I_{\a+n}}$-basis of $M_{I_{\a+n}}$.
Then 
$$\aligned 
      & M/I_{\a+n}M\cong 
M_{I_{\a+n}}/I_{\a+n}M_{I_{\a+n}}\\
    = & (R_{I_{\a+n}}v_1\oplus\cdots \oplus R_{I_{\a+n}}v_r)/I_{\a+n}(R_{I_{\a+n}}v_1\oplus \cdots\oplus R_{I_{\a+n}}v_r)\\
\cong & \bigoplus_{i=1}^r R_{I_{\a+n}}v_i/I_{\a+n}(R_{I_{\a+n}}v_i)
\cong  (R/I_{\a+n})^r\endaligned$$ as $R$ modules. In particular, we have $\dim M/I_{\a+n}M=r$.

Considering the cuspidal module $\mathfrak{W}^{(\a)}=\bigoplus_{m\in \Z^d} (M/I_{\a+m}M)\otimes x^m$ over $W_d$, we have $\dim M/I_{\alpha}M=\dim M/I_{\a+n}M=r$.
\end{proof}

From this lemma, we have

\begin{corollary}\label{free-3}\begin{itemize}\item[(1).] Any  $\K_d$-module that is a  finitely generated $U(\mh)$-module is a free $U(\mh)$-module.
\item[(2).] Any admissible $\widetilde W_d$-module that is a  finitely generated $U(\mh)$-module is a free $U(\mh)$-module.
\item[(3).] A nontrivial simple $W_d$-module $M$ that is a finitely generated $U(\mh)$-module is a free $U(\mh)$-module if and only if $\dim M/I_0M=\rank M$. \end{itemize}\end{corollary}

\begin{proof} (1) follows from (2), since any $\K_d$-module is automatically an admissible $\widetilde{W}_d$-module, as we remarked in Section 2. And (3) follows directly from Lemma \ref{L-free} and Lemma \ref{lemma-dim}. 

For (2), we take an admissible $\widetilde{W}_d$-module $M$, then $x^nM=M$ and $x^nI_0M=I_{n}M$, 
which implies $M/I_0M\cong M/I_nM=\rank M$ for any $n\in\Z^d$. Again the result follows from Lemma \ref{L-free} and Lemma \ref{lemma-dim}.
\end{proof}

Let $P$ be an admissible $\widetilde{W}_d$-module.
Then we have the $W_d$-module homomorphisms for $k=1,2,\cdots, d$,
\begin{equation*}\begin{array}{lrcl}
\pi_{k-1}:& \F(P,V(\delta_{k-1},k-1) & \rightarrow & \F(P, V(\delta_{k},k)),\\
      & y\otimes v & \mapsto & \sum\limits_{j=1}^{d} \partial_jy\otimes e_j\wedge v,
\end{array}\end{equation*}
for all $y\in P$ and $v\in V(\delta_{k-1}, k-1)$. Note that the definition for $\phi_{k-1}$ has different form from that in \cite{LLZ}, but they are essentially same using Theorem 2.3.
Denote $ \L_d(P,k)=\text{Im}\ \pi_{k-1}$ and $\tilde \L_d(P,k)=\text{Ker}\ \pi_{k}$,
where $\pi_d=0$ and $\tilde{\L}_d(P,d)=\F(P, V(\delta_{d},d))$.

Thanks to the isomorphism between $F(P,V)$ and $\F(P,V)$ in Theorem 2.3 we can collect some results
on these modules from \cite{LLZ}. If  $k=1,\cdots,d$, $P$ is a simple $\K_d$-module and finitely generated over $U(\mh)$, then
\begin{itemize}\item[(1).] $\F(P, V(\delta_{0},0))$ is simple, 
\item[(2).] $\F(P, V(\delta_{k},k))$ is not simple, 
\item[(3).] $\L_d(P,k)$ is simple, 
\item[(4).] $\L_d(P,k)\subseteq \tilde \L_d(P,k)$, 
\item[(5).] $\tilde{\L}_d(P,k)/\L_d(P,k)$ is trivial,  
\item[(6).] $\F(P, V(\delta_{d},d))/\L_d(P,d)$ is trivial.\end{itemize}
Here the only   thing we might need to explain is that, since $P$ is a free $U(\mh)$-module (Lemma \ref{free-3}(1)),
we have $\sum_{i=1}^d \partial_iP\ne P$ and $\F(P, V(\delta_d,d))$ is not simple by
Corollary 3.6 in \cite{LLZ}.

Recall from Lemma \ref{p3.5} that $$\mathfrak{W}(\F(P,V(\delta_{k},k)))\cong \F(\mathfrak{W}(P),V(\delta_{k},k)).$$
We have the $W_d$-module homomorphism 
\begin{equation*}\aligned 
\mathfrak{W}(\pi_{k-1}):\ \F(\mathfrak{W}(P),V(\delta_{k-1},k-1)) &\rightarrow \F(\mathfrak{W}(P), V(\delta_{k},k))\\
                    (y+I_nP)\otimes x^n\otimes v &\mapsto \sum\limits_{j=1}^{d} (\partial_jy+I_nP)\otimes x^n\otimes e_j\wedge v,
\endaligned\end{equation*}
for all $y\in P, v\in V(\delta_{k-1}, k-1)$ and $n\in \Z^d$. 

\begin{lemma}\label{lemma-main3} Let notations be as above except that $P$ is a simple 
$\K_d$-module which is finitely generated $U(\mh)$-module of rank $r$.  For all $i=1,2,\ldots,d$,  
\begin{itemize}
\item[(1).] $\rank \L_d(P,i)=\rank \tilde \L_d(P,i)=r{{d-1}\choose i-1}$;
\item[(2).] $\rank(\L_d(P,i)/I_0\L_d(P,i))=r{{d}\choose i-1}$;
\item[(3).] $\L_d(P,1)\cong \F(P,V(0,0))$ is a free $U(\mh)$-module;
\item[(4).] $\L_d(P,i)$ are not   free $U(\mh)$-modules.
\end{itemize}
\end{lemma}

\begin{proof}(1). Recall that $\F(P, V(\delta_i,i))$ is a free $U(\mh)$-module. Hence $\L_d(P,i)$ and $\tilde{\L}_d(P,i)$ are both torsion free over $U(\mh)$ and of the same rank, since $U(\mh)\tilde{\L}_d(P,i)\subset  {\L}_d(P,i)$. Moreover from $$\F(P, V(\delta_i,i))/\tilde{\L}_d(P,i)=\L_d(P,i+1),\ \forall\ i=1,\ldots,d-1,$$ we have \begin{equation}\label{eq3}\rank \L_d(P,i)+\rank \L_d(P, i+1)=\rank \F(P, V(\delta_i,i))=r{{d}\choose i}.\end{equation} Recall that $\L_d(P,1)\cong \F(P,V(0,0))$. Then $\rank \L_d(P,1)=r$. Now from (\ref{eq3}) we   deduce $\rank \L_d(P,i)=r{{d-1}\choose i-1}$ for all $ i=1,2,\ldots,d$ by induction on $i$.

(2). Let $\{w_1,\cdots,w_r\}$ be a $U(\mh)$-basis of $P$. Note that  
$$\L_d(P,i)\hskip -3pt =\hskip -3pt 
\span\left\{\sum_{j=1}^d \partial_j w\otimes e_j\wedge v\,\Big|\,w\in P, v\in \bigwedge^{i-1} \C^d\right\}.$$
$$I_{0}\L_d(P,i)\hskip -3pt =\hskip -3pt 
\span\left\{\sum_{j=1}^d \partial_j w\otimes e_j\wedge v\,\Big|\,w\in I_{0}P, v\in \bigwedge^{i-1} \C^d\right\}.$$
By considering the total degree on $\partial_1,\ldots, \partial_d$, it is easy to deduce the following vector space decomposition $\L_d(P,i)=X\oplus I_0\L_d(P,i)$, where
\begin{equation}
X\hskip -3pt =\hskip -3pt \span\hskip -3pt \left\{\hskip -3pt \sum_{j=1}^d \partial_j w_k\otimes e_j\wedge v\,\Big|\,v\in \bigwedge^{i-1} \C^d\hskip -3pt, k=1,2,\ldots, r\hskip -3pt \right\}.
\end{equation}
Take any basis $v_1,\cdots,v_s$ of $\bigwedge^{i-1} \C^d$, where $s={{d}\choose i-1}$. 
It is easy to check that $\sum_{j=1}^d \partial_j w_k\otimes e_j\wedge v_i, i=1,\cdots,s, k=1,\cdots,r$, 
form a basis of $X$. Hence $\dim \L_d(P,i)/I_0\L_d(P,i)=\dim X=r{{d}\choose i-1}$.

(3) is clear. (4) follows from (1) and (2) by Lemma \ref{free-3} (3). 
\end{proof}
  
\begin{lemma}\label{trivial} Let $P$ be a simple $\K_d$-module which is a finitely generated $U(\mh)$-module. 
Then as a $W_d$-module,
\begin{itemize}
\item[(a).] $\F(P, V(\delta_k,k))$ has finite composition length for $k=1,\cdots,d$;
\item[(b).] $\F(P, V(\delta_d,d))$ has a unique minimal submodule $\L_d(P, d)$ and the quotient
$\F(P, V(\delta_d,d))/\L_d(P, d)$ is trivial;
\item[(c).] $\F(P, V(\delta_k,k))$ has a unique minimal submodule $\L_d(P,k)$, a unique maximal
submodule $\tilde{\L}_d(P,k)$, the quotient $\tilde{\L}_d(P,k)/\L_d(P,k)$ is trivial, 
and $\F(P, V(\delta_k,k))/\tilde{\L}_d(P,k)\cong\L_d(P,k+1)$ for $k=1,\cdots\hskip-1pt,\hskip-1pt d-1$.
\end{itemize}
\end{lemma}

\begin{proof}
(a). First note that $\F(P,V(\delta_k, k))$ has finitely many trivial simple $W_d$-subquotient since it is finitely generated over $U(\mh)$-modules. Then the result follows from the fact that any nontrivial simple $W_d$-sub-quotient of $F(P,V(δk,k))$ is $U(h)$-torsion free.

Then we note that $\F(P, V(\delta_k,k))$ has no trivial $W_d$-submodules,
since any submodule of a $U(\mh)$-torsion free module $\F(P, V(\delta_k,k))$ must be $U(\mh)$-torsion free. 

(b) follows from the above argument for $k=d$ and the facts $\L_d(P,d)$ is simple and $\F(P,V(\delta_d,d))/\L_d(P,d))$ is trivial.

(c) Now take $1\leq k\leq d-1$. First suppose that $N$ is a nontrivial simple submodule of $\F(P, V(\delta_k,k))$. To the contrary, we assume that  $N\neq\L_d(P,k)$.
We have the submodule $N\oplus  \L_d(P,k)$ of  $\F(P, V(\delta_k,k))$
and $N\cong \L_d(P,k+1)$. By the definition of $\tilde \L_d(P,k)$ we deduce $\F(P, V(\delta_k,k))=$ $N\oplus \tilde  \L_d(P,k)$ as $W_d$-modules. Note that $\mathfrak{W}(P)\cong A_d^r$ as $\K_d$-modules. So
\begin{equation}\label{long}\aligned &\mathfrak{W}(N)\oplus\mathfrak{W}(\tilde  \L_d(P,k))=\mathfrak{W}(\F(P, V(\delta_k,k)))\\
\cong &\F(\mathfrak{W}(P), V(\delta_k,k))\cong \F(A_d, V(\delta_k,k))^r. \endaligned\end{equation}
 as admissible ${W}_d$-modules. 
 
By definitions, we have the $W_d$-module epimorphism
\begin{equation}\label{epi5.1}
\phi:\quad \mathfrak{W}(N)\cong\mathfrak{W}(\L_d(P,k+1))\to 
\L_d(\mathfrak{W}(P),k+1)
\end{equation}
given by assigning the element
$\big(\sum\limits_{j=1}^{d} \partial_jy\otimes e_j\wedge v+I_n\L_d(P,k+1)\big)\otimes x^n$ to 
$ \sum\limits_{j=1}^{d}(\partial_jy+I_nP)\otimes x^n\otimes e_j \wedge v$.
Now from Lemma \ref{lemma-main3} and Lemma \ref{lemma-dim}, we have 
$$\dim \L_d(P,k+1)/I_n\L_d(P,k+1)=r{{d-1}\choose k},\ \forall\ n\in\Z^d\setminus\{0\},$$ 
and it is also known that 
$$\dim \L_d(\mathfrak{W}(P), k+1))_n=r{{d-1}\choose k},\ \forall\  n\in\Z^d\setminus\{0\}.$$ 

Thus $\ker\phi$ in \eqref{epi5.1} is a trivial submodule of $\mathfrak{W}(N)$. 
On the other hand, $\F(A_d, V(\delta_k,k))$ does not have any nonzero trivial submodule. 
So $\ker(\phi)=0$ by \eqref{long} and $\mathfrak{W}(N)\cong\L_d(\mathfrak{W}(P),k+1)\cong\L_d(A_d,k+1)^r$, 
which implies $\L_d(A_d,k+1)$ is a direct summand of $\F(A_d, V(\delta_k,k))$, impossible since $\F(A_d, V(\delta_k,k))$ is indecomposible.
So we must have $N=\L_d(P,k)$.

Now suppose that $N'$ is a maximal $W_d$-submodule of $\F(P, V(\delta_k,k))$. We know that ${\L}_d(P,k)\subset N'$.
If $N'\neq \tilde{\L}_d(P,k)$, then $N'+\tilde{\L}_d(P,k)=\F(P, V(\delta_k,k))$ and 
$$\F(P, V(\delta_k,k))/N'\cong \tilde{\L}_d(P,k)/N'\cap\tilde{\L}_d(P,k)$$ 
is trivial, since $\L_d(P,k)\subseteq N'\cap\tilde{\L}_d(P,k)$. 
Take any $w\in \F(P, V(\delta_k,k))\setminus N'$. Then there exists 
$n\in\Z^d$ such that $w\notin I_n\F(P, V(\delta_k,k))$. Thus $w+I_n\F(P, V(\delta_k,k))$ is nonzero in $\mathfrak{W}(\F(P, V(\delta_k,k)))\cong \F(A_d, V(\delta_k,k))^r$ and gives rise to a 
nonzero trivial quotient module of $\F(A_d, V(\delta_k,k))^r$, impossible.
So we must have $N'=\tilde{\L}_d(P,k)$.
\end{proof}

\begin{corollary} \label{cor5.5} Let $M$ be a simple  $W_d$-module that  is a free $U(\mh)$-module of finite rank when restricted to $U(\mh)$.
Then $M$ is isomorphic to $\F( P, V)$ for a simple  $\K_d$-module $P$ that is  a a free $U(\mh)$-module module of finite rank, and a finite dimensional  simple  $\gl_d$-module $V$ which is not isomorphic to $V(\delta_k,k)$ for any $k=1,2,\ldots,d$.
\end{corollary}
\begin{proof}It follows directly from Theorem \ref{thm5.5}, Lemmas \ref{lemma-main3} and \ref{trivial}. \end{proof}

\section{Simple $\K_d$-modules that are finitely generated $U(\mh)$-modules}

From Corollary \ref{free-3} we know that any  $\K_d$-module that is a  finitely generated $U(\mh)$-module is a free $U(\mh)$-module of finite rank. In this scetion we will characterize such  $\K_d$-module.
Let \begin{equation}\label{polynomail}
f_i=1+\sum_{j=1}^{n_i}a_{ij}x_i^j\in \K_d,\quad i=1,2\ldots,d
\end{equation} where $n_i\ge 1, a_{ij}\in U(\mh),  a_{in_i}\in \C^\star$. 
Define the $\K_d$-module $$S_{f_1,\ldots,f_d}=\K_d\Big/\Big(\sum_{i=1}^d \K_df_i\Big).$$

\begin{lemma} 
\begin{itemize}
\item[(1).] Any quotient $\K_d$-module of $S_{f_1,\ldots,f_d}$ is a finitely generated $U(\mh)$-module.
\item[(2).] As $\K_d$-module, $S_{f_1,\ldots,f_d}$ has finite composition length.
\item[(3).]  Any simple $\K_d$-module that is a finitely generated $U(\mh)$-module is isomorphic to a quotient of  $S_{f_1,\ldots, f_d}$ for some $f_i\in \K_d$ of the form in (6.1).
\end{itemize}
\end{lemma}

\begin{proof}  (1).  It is easy to see that
$$\K_d=\sum_{i=1}^{d}\K_df_i+\sum_{0\le r_i\le n_i-1} U(\mh)x^r.$$
So $S_{f_1,\ldots,f_d}$ is a finitely generated $U(\mh)$-module, and hence any quotient $\K_d$-module of of $S_{f_1,\ldots,f_d}$ is a finitely generated $U(\mh)$-module.

(2) follows from Corollary 4.9.

(3). Now suppose that $M$ is a  simple $\K_d$-module that is a finitely generated $U(\mh)$-module when restricted to $U(\mh)$. Take a nonzero vector $v\in M$. For any $j=1,\cdots,d$, consider the $U(\mh)$-submodule $M_j$ of $M$  generated by $\C[x_j^{\pm1}]v$ which is finitely generated as a $U(\mh)$-module.
There exists some $s,k\in \Z$ such that $M_j=\sum_{i=s}^k U(\mh)x_j^i v$.  
Since $x_j^{s-1}v, x_j^{k+1}v\in M$, we can find $h_i, g_i\in U(\mh)$ such that
$$x_j^{s-1}v=\sum_{i=s}^k h_i(\partial)x_j^iv, \ \ \ \  x_j^{k+1}v=\sum_{i=s}^k g_i(\partial)x_j^iv.$$ 
Clearly $(x_j^{s-1}-\sum_{i=s}^k(h_i+g_i)x_j^i+x_j^{k+1})v=0$. We can take
  $$f_j(x)=x_j^{1-s}(x_j^{s-1}-\sum_{i=s}^k(f_i+g_i)x_j^i+x_j^{k+1}),\ \ \ j=1,\cdots,d.$$
 Thus  $M$ is isomorphic to a simple quotient of $S_{f_1,\ldots, f_d}$.
\end{proof}

\begin{example} Let $k=1, 2, \cdots, d, A^{(k)}=\C[x_k^{\pm1}]$ and $\K^{(k)}=\C[x_k^{\pm 1}, \partial _k]$. Let $f_k=\partial_k-g_k(x_k)$, where $g_k(x_k)=\sum\limits_{i=-m_k}^{n_k} a_{k,i}x_k^i \in A^{(k)}$ with
$a_{k,i}\in\C$, $a_{k,m_k}a_{k,n_k}\ne 0$ and $m_k, n_k>0$. Then we have the  $\K^{(k)}$-module $S^{(k)}_{f_k}=\K^{(k)}/\K^{(k)}f_k\cong A^{(k)}$ with the actions:
 $$x^i x_k^l=x_k^{i+l} , \  \  \partial_kx_k^l=x_k^l(l+g(x_k ) ),\ \forall\ i,l\in\Z.$$
From Theorem 12 (3) in \cite{LGZ} we know that   $S^{(k)}_{f_k}$ is a simple module over $\K^{(k)}.$
 It is easy to see that $S^{(k)}_{f_k}$ is a finitely generated $\C[\partial_k]$-module. Moreover, $S^{(k)}_{f_k}$ is a free $\C[\partial_k]$-module of rank $m_k+n_k$. 
In particular, we have simple $\K^{(k)}$-module that is a free $\C[\partial_k]$-module of any given positive rank. 
 
 By using the tensor product, we have simple $\K_d$-modules $S_{f_1, f_2, \cdots, f_k}$ that are free $U(\mh)$-modules of any given positive rank. Using Corollary \ref{cor5.5} we have simple $W_d$-module that is a free $U(\mh)$-module of any given positive rank.
 \end{example}

\

\begin{center}
\bf Acknowledgments
\end{center}

\noindent 
X.G. is partially supported by NSF of China (Grant 11471294) and the Outstanding Young Talent Research Fund of Zhengzhou University (Grant 1421315071); 
G.L. is partially supported by NSF of China (Grant 11301143) and  the grants  of Henan University (2012YBZR031, 0000A40382); 
R.L. was partially supported by the NSF of China (Grant 11471233, 11371134); 
K.Z. is partially supported by  NSF of China (Grant 11271109) and NSERC.

\

\

 \noindent X.G.:  School of Mathematics and Statistics, Zhengzhou University,
Zhengzhou, P. R. China.  Email: guoxq@zzu.edu.cn 

\vspace{0.2cm}  \noindent G.L.: College of Mathematics and Information
Science, Henan University, Kaifeng 475004, China. Email:
liugenqiang@amss.ac.cn

\vspace{0.2cm}\noindent R.L.: Department of Mathematics, Soochow University, Suzhou, P. R. China.  Email: rencail@amss.ac.cn

\vspace{0.2cm} \noindent K.Z.:     College of
Mathematics and Information Science, Hebei Normal (Teachers)
University, Shijiazhuang, Hebei, 050016 P. R. China, and Department of Mathematics, Wilfrid
Laurier University, Waterloo, ON, Canada N2L 3C5. Email:
kzhao@wlu.ca

\end{document}